\documentclass[11pt]{article}
\usepackage[margin=1in]{geometry}

\usepackage{hyperref}
\hypersetup{
    colorlinks=true,
    linkcolor=blue,
    filecolor=magenta,      
    urlcolor=cyan,
    citecolor=red
}

\usepackage{amsmath,amsthm,amssymb}

\usepackage{tikz-cd}
\usepackage{pgf,tikz}
\usepackage{pgfplots}
\pgfplotsset{compat=1.11}
\usepgfplotslibrary{fillbetween}
\usetikzlibrary{intersections}
\pgfdeclarelayer{bg}
\pgfsetlayers{bg,main}

\usepackage{mathrsfs}
\usetikzlibrary{arrows}

\title{Spherical metrics with conical singularities on 2-spheres}
\author{Subhadip Dey}
\date{}
\newcommand{\Address}{{
  \bigskip
  \footnotesize

  \textsc{Department of Mathematics}
    \par\nopagebreak
  \textsc{University of California, Davis}
    \par\nopagebreak
  \textsc{1 Shields Ave, Davis, CA 95616}
    \par\nopagebreak
  \textit{E-mail address}: \texttt{sdey@math.ucdavis.edu}
  }}

\begin{document}
\maketitle
\theoremstyle{plain}
\newtheorem{thm}{Theorem}
\newtheorem{lem}[thm]{Lemma}
\newtheorem{prop}[thm]{Proposition}
\newtheorem{cor}[thm]{Corollary}
\newtheorem{ques}{Question}

\theoremstyle{definition}
\newtheorem{defn}[thm]{Definition}
\newtheorem{conj}[thm]{Conjecture}
\newtheorem{exmp}[thm]{Example}

\theoremstyle{remark}
\newtheorem{rem}[thm]{Remark}
\def\t{\theta}\def\T{\pmb{\theta}}

\begin{abstract}
Suppose that $\theta_1,\theta_2,\dots,\theta_n$ are positive numbers and $n\ge 3$. We want to know whether there exists a spherical metric on $\mathbb{S}^2$ with $n$ conical singularities of angles $2\pi\theta_1,2\pi\theta_2,\dots,2\pi\theta_n$. A sufficient condition was obtained by Gabriele Mondello and Dmitri Panov \cite{mondello2016spherical}. We show that their condition is also necessary when we assume that $\theta_1,\theta_2,\dots,\theta_n \not\in \mathbb{N}$.
\end{abstract}

\section{Introduction}
A classical question in the theory of Riemann surfaces asks which real functions $f$ on a Riemann surface $S$ are equal to the curvature of a pointwise conformal metric. For simplicity, all the surfaces considered in this section are orientable closed 2-manifolds. If $S$ has genus $\ge 2$, then Berger \cite{berger1971riemannian} showed that any smooth negative function is the curvature of a unique conformal metric. In the case of a Riemann surface of genus $1$, i.e. when $S$ is a torus with a flat metric $g$, Kazdan and Warner \cite{kazdan1974curvature} proved that a function $f (\not\equiv 0): S \rightarrow \mathbb{R}$ is the curvature of a metric in the conformal class of $g$ if and only if $f$ changes sign and satisfies $\int_S f dA <0$, where $dA$ is the area form  $g$.

These results have been generalized by Marc Troyanov in the case of surfaces with singularities of a special type, called \emph{conical singularities} which we shall define below. Let $S$ be a surface. A \emph{real divisor} $\pmb{\beta}$ on $S$ is a formal sum
\[ \pmb{\beta} = \beta_1 x_1 + \dots + \beta_n x_n,\]
where $x_i \in S$  are pairwise distinct and $\beta_i$ are real numbers. For the pair $(S,\pmb{\beta})$, define 
\[\chi(S,\pmb{\beta}) = \sum_{i=1}^n\beta_i + \chi(S),\]
called the \emph{Euler characteristic} of $(S,\pmb{\beta})$. 

Suppose that $\pmb{\beta} = \beta_1 x_1 + \dots + \beta_n x_n$ is a real divisor on $S$ such that $\beta_i > -1$. Let $g$ be a Riemannian metric on $S$ defined away from $x_1,\dots,x_n$ such that each point $x_i$ has a neighborhood $U_i$ in $S$ with coordinate $z_i$ satisfying $z_i(x_i) = 0$ on which $g$ has the following form,
\begin{align*}\label{angle}
ds^2 = e^{2u_i} |z_i|^{2\beta_i}|dz_i|^2.
\end{align*}
Here $u_i : U_i\rightarrow \mathbb{R}$ is a continuous function such that $u_i|_{U_i - \{x_i\}}$ is differentiable (of class at least $C^2$). The point $x_i$ is called a \emph{conical singularitiy} of the metric $g$ of angle $2\pi(\beta_i+1)$. We refer to this type of metrics $g$ as \emph{(Riemannian) metrics with conical singularities}. We say that the metric $g$ (with conical singularities) \emph{represents} the divisor $\pmb{\beta}$.

In this terminology, Troyanov  \cite{troyanov1991prescribing}  proved the following two theorems.

\begin{thm} [\cite{troyanov1991prescribing}] \label{troyanov1}
Let $S$ be a Riemann surface with a real divisor $\pmb{\beta} = \beta_1 x_1 + \dots + \beta_n x_n$ such that $\beta_i > -1$, $i=1,\dots,n$. If $\chi(S,\pmb{\beta}) <0$, then any smooth negative function on $S$ is the curvature of a unique conformal metric which represents $\pmb{\beta}$.
\end{thm}

\begin{thm} [\cite{troyanov1991prescribing}] \label{troyanov2}
Let $S$ be a Riemann surface with a real divisor $\pmb{\beta} = \beta_1 x_1 + \dots + \beta_n x_n$ such that $\beta_i > -1$, $i=1,\dots,n$. If $\chi(S,\pmb{\beta}) = 0$, then any function $f (\not\equiv 0): S \rightarrow \mathbb{R}$ is the curvature of a conformal metric representing $\pmb{\beta}$  if and only if $f$ changes sign and satisfies $\int_S f dA <0$, where $dA$ is the area element of a conformally flat metric on $S$ with singularities.
\end{thm}

Theorem \ref{troyanov1} generalizes Berger's result, and Theorem \ref{troyanov2} generalizes Kazdan and Warner's result.
 Restricting our attention only to constant curvature metrics with conical singularities, we formulate the following question.

\begin{ques}\label{q1}
Let $S$ be a Riemann surface with a real divisor $\pmb{\beta} = \beta_1 x_1 + \dots + \beta_n x_n$ such that $\beta_i > -1$, $i=1,\dots,n$. Can $\pmb{\beta}$ be represented by a conformal metric of constant curvature?
\end{ques}

The case when $\pmb{\beta} = 0$ is completely understood due to the classical uniformization theorems. If $\chi(S)\ge 0$, then any conformal class has a representative of constant curvature.  When $\chi(S)<0$, existence and uniqueness of such a conformal metric is provided by classical uniformization theorems of Koebe and Poincar\'e.

\medskip
When $\pmb{\beta} \ne 0$, the Theorems \ref{troyanov1} and \ref{troyanov2} give complete understanding in the case when $\chi(S,\pmb{\beta})\le 0$. See also the work of McOwen \cite{mcowen1988point} for the $\chi(S,\pmb{\beta}) < 0$ case.

\medskip
Here we focus on the particular case when $S$ is the Riemann sphere, which has been least understood. Given a real divisor $\pmb{\beta} =  \beta_1 x_1 + \dots + \beta_n x_n$ with $\beta_i>-1$, $i=1,\dots,n$, does there exist a conformal metric $g$ with conical singularities on $S$ which represents $\pmb{\beta}$? If we further assume that $g$ has constant curvature $1$, then Gauss-Bonnet Theorem gives a restriction on $\pmb{\beta}$, namely
\[\chi(S,\pmb{\beta}) > 0.\]
But further restrictions on $\pmb{\beta}$ were found in addition to this in the case $n=2,3$. When $n=2$, we simply need to require that $\beta_1 = \beta_2$ by the work of Troyanov \cite{troyanov1989metrics}. When $n=3$, by Eremenko's result in  \cite{eremenko2004metrics}, a conformal metric with conical singularities exists if and only if some inequalities on the numbers $\beta_1,\beta_2,\beta_3$ are satisfied.  Moreover, in this case,
there exists a \emph{spherical triangle} with angles $\pi(\beta_1+1), \pi(\beta_2+1), \pi(\beta_3+1)$. Here a {spherical triangle} means a topological closed disk with a metric of constant curvature 1 such that the boundary is a piecewise geodesic loop with three singular points. Examples include geodesic triangles immersed in ordinary sphere $\mathbb{S}^2$. A sphere with three conical singularities is the double of such a triangle. When $n>3$, the situation becomes more complicated.

Therefore, before answering the Question \ref{q1} for higher $n$'s, perhaps one needs to have a complete list real divisors $\pmb{\beta}$ which may be represented by a spherical metric $g$ on a sphere with conical singularities. Here \emph{spherical} means that the metric has constant curvature $1$. Note that the divisor $\pmb{\beta}$ can be represented by a spherical metric with conical singularities if and only if there exists a spherical metric  with $n$ conical singularities of angles $2\pi(\beta_1+1), \dots, 2\pi(\beta_n +1)$.

\medskip
Let $\mathbb{R}_+^n$ be the set of all points in $\mathbb{R}^n$ with positive coordinates. A point $\pmb{\theta} = (\theta_1,\dots,\theta_n)\in \mathbb{R}_+^n$ is called \emph{admissible} if there exists a sphere $S$ with a spherical metric $g$ with $n$ conical points $x_1, \dots, x_n$ of angles $2\pi\theta_1,\dots,2\pi\theta_n$, respectively. In this case, the real divisor $\pmb{\beta} = (\theta_1-1)x_1 + \dots + (\theta_n -1)x_n$ is represented by $g$. By abuse of notation, we write $\pmb{\theta} - \pmb{1} = \pmb{\beta}$, where $\pmb{1} = (1,\dots,1)$.

\begin{ques}
Which points $\pmb{\theta} \in \mathbb{R}^n_+$ are admissible?
\end{ques}

A major progress was done to answer this question by Mondello and Panov \cite{mondello2016spherical}, as stated in the next two theorems.

\begin{thm} [\cite{mondello2016spherical}]\label{panov1}
If $\pmb{\theta} \in \mathbb{R}^n_+$ is admissible, then
\begin{align}	
	\chi(S, \pmb{\theta} -\pmb{1}) &> 0, \tag{P} \label{P}\\
	d_1(\mathbb{Z}_o^n, \pmb{\theta}-\pmb{1}) &\ge 1,	\tag{H}\label{H}
\end{align}
where $d_1$ is the standard $\ell^1$ distance on $\mathbb{R}^n$ and $\mathbb{Z}_o^n$ is the set of all points $\pmb{v} \in \mathbb{R}^n$ with integer coordinates such that $d_1(\pmb{v},\pmb{0})$ is odd. 
Moreover, if the equality in (H) is attained, then the holonomy of the a metric which corresponds to $\pmb{\theta}$ is coaxial.
\end{thm}

Here we clarify what we mean by saying that the holonomy of a metric with conical sigularities is coaxial. Note that a metric $g$ on $S$ with conical singularities is actually defined as a Riemannian metric only on $S - \{\text{conical sigularities}\}$. Thus we have a holonomy representation of the metric $g$ which is a homomorphism $\phi : \pi_1(S - \{\text{conical sigularities}\}) \rightarrow SO(3)$ from the fundamental group of $S - \{\text{conical sigularities}\}$ to the group of rotations of the standard sphere $\mathbb{S}^2$. The metric $g$ is said to have \emph{coaxial holonomy} if the image of $\phi$ is contained in an one-parameter subgroup of $SO(3)$.

\medskip
The following is a partial converse to Theorem \ref{panov1}.

\begin{thm} [\cite{mondello2016spherical}] \label{panov2}
If $\pmb{\theta} \in \mathbb{R}^n_+$ satisfies the positivity constraints (\ref{P}) and the holonomy constraints (\ref{H}) strictly, then $\pmb{\theta}$ is admissible. Moreover, each metric corresponding to $\pmb{\theta}$ has non-coaxial holonomy.
\end{thm}

Theorems \ref{panov1} and \ref{panov2} classified all the admissible points which do not satisfy (H'), where (H') is the equality in (\ref{H}),
\[ d_1(\mathbb{Z}_o^n,\pmb{\theta}-\pmb{1}) = 1,	\tag{H'}\label{H'} \]
but did not provide answer in the complementary case. We formulate this case in the following question.

\begin{ques}
For $n\ge 4$, which points in $\mathbb{R}_+^n$ satisfying (\ref{P}) and (\ref{H'}) together are admissible?
\end{ques}

In Theorem \ref{thm:main}, by restricting to the ``non-integral'' case, we give an answer to this question. We show that when $\theta_1,\dots,\theta_n \not\in \mathbb{N}$, the sufficient conditions in Theorem \ref{panov2} are actually necessary.

The ``integral'' case has been analyzed by Kapovich \cite{kapovich} who showed that a spherical metric with conical singularities of angles $2\pi\theta_1,\dots,2\pi\theta_n$, where $\theta_1,\dots,\theta_n \in \mathbb{N}$, exists if and only if the integers $\theta_1,\dots,\theta_n$ satisfy (\ref{H'}) together with the polygon inequality
\[(\theta_i - 1) \le \frac{1}{2}\sum_{j= 1}^n (\theta_j-1), \quad\text{for all }0\le i\le n.\]

A well-formulated collection of necessary and sufficient conditions on the complement of these two cases (the ``mixed'' case) is unknown to the author.

\subsection*{Acknowledgement}  I thank my advisor, Prof. Michael Kapovich, for introducing this problem to me. I am grateful for his constant encouragement, helpful comments and numerous discussions on this project. I also thank an anonymous referee for commenting on an earlier version of this paper which helped to improve the proof of the main result.

\section{Main Result}
The aim of this paper is to establish the following theorem.

\begin{thm}\label{thm:main}
Suppose that $\pmb{\theta} = (\theta_1,\dots, \theta_n)\in\mathbb{R}_+^n$ satisfies (P) and (H') and $n\ge 3$. If $\theta_i \not\in \mathbb{N}$ for all $1\le i\le n$ then $\pmb{\theta}$ is not admissible.
\end{thm}

The case $n=3$ has already been treated in \cite{eremenko2004metrics}. 

\medskip
By a \emph{singular spherical surface}, we mean a surface, possibly with boundary, with a spherical metric $g$ with a set $X$ of conical singularities.  The points in $X$ are called \emph{singular points} and the points in $S - (\partial S \cup X)$ are called \emph{regular points}. We further assume that the boundary, if nonempty, is a union of smooth curves. A path $\gamma$ in $S$ is called a \emph{geodesic arc} if it's restriction to the complement of $X$ is a connected geodesic arc in the Riemannian sense. We allow geodesic arcs to have singular endpoints. If a geodesic arc has same initial and final point, we call it a \emph{geodesic loop}. A composition of piecewise geodesic arcs is called a \emph{peicewise geodesic path}. The metric $g$ induces a distance function $d_S$ defined by
\[ d_S(x',x'') = \text{infimum over the length of piecewise geodesic paths connecting } x'\text{ and } x''.\]
We prove,

\begin{thm}\label{thm:general}
Suppose that $S$ is a singular spherical surface without boundary with a discrete set of conical singularities $X$. Suppose that the metric  $d_S$ is complete. For $x\in X$, if $S$ has no geodesic loop based at $x$ of length shorter than $\pi$, then
\[\min_{x\ne x'\in X}\{d_S(x,x')\} \le \pi.\]
If the equality occurs, then $S$ is compact and $|X| \le 2$.
\end{thm}

Combining Theorem \ref{thm:main} with Theorems \ref{panov1} and \ref{panov2}, we get necessary and sufficient conditions on $(\theta_1,\dots,\theta_n)$, provided $\theta_i \not\in \mathbb{N}$, for which there exists a sphere $S$ with $n\ge 3$ conical singularities of angles $2\pi\theta_1,\dots,2\pi\theta_n$. The analogous case when $n\le 2$ is known from \cite{troyanov1989metrics}, as discussed in the previous section.

\begin{thm}
Suppose that $\pmb{\theta} = (\theta_1,\dots, \theta_n)\in\mathbb{R}_+^n$ where $n\ge 3$. If $\theta_i\not\in \mathbb{N}$, $1\le i\le n$, then $\pmb{\theta}$ is admissible if and only if 
\begin{align*}	
	\chi(S, \pmb{\theta} -\pmb{1}) &> 0, \\
	d_1(\mathbb{Z}_o^n, \pmb{\theta}-\pmb{1}) &> 1,
\end{align*}
where $d_1$ is the standard $\ell^1$ distance on $\mathbb{R}^n$ and $\mathbb{Z}_o^n$ is the set of all points $\pmb{v} \in \mathbb{R}^n$ with integer coordinates such that $d_1(\pmb{v},\pmb{0})$ is odd. 
\end{thm}

\section{Proof of Theorem \ref{thm:general}}

Let $\Sigma$ be a closed, genus zero, singular spherical surface with two conical singularities $y$ and $y'$ of angle $\theta$. Such a surface has a shape of a football, which we call a ``Troyanov's football''. We denote the closed $r$-neighborhood of $y$ by $\Sigma_r$. For small $r$, $\Sigma_r$ is a model neighborhood of a conical singularity of angle $\theta$. Note that when $r\ge\pi$, $\Sigma_r = \Sigma$.

\begin{defn}
Let $f:S_1\rightarrow S_2$ be a map between two singular spherical surfaces. Such a map $f$ is called a \emph{locally isometric} map (or a \emph{local isometry}) between these surfaces if the following conditions are satisfied.
\begin{enumerate}
\item $f$ maps the singular (resp. regular) points of $S_1$ to the singular (resp. regular) points of $S_2$.
\item $f$ is a local isometry on $S_1 - (\partial S\cup X)$ in the Riemannian sense.
\item For each boundary point $s_1\in \partial S_1$, there exists a neighborhood $U$ of $s_1$, a neighborhood $V$ of $s_2$, isometries $\phi_1:U\rightarrow\mathbb{S}^2$, $\phi_2:V\rightarrow\mathbb{S}^2$ and $\hat{f}\in \text{Isom}(\mathbb{S}^2)$ such that the following square commutes:
\[\begin{tikzcd}
U \arrow[r, "f"] \arrow[d, "\phi_1"] &V\arrow[d, "\phi_2"]\\
\mathbb{S}^2 \arrow[r, "\hat{f}"] & \mathbb{S}^2
\end{tikzcd}\]
\end{enumerate}
$f$ is called an \emph{isometric embedding} if it is an injective locally isometric map.
\end{defn}

Let $S$ be a singular spherical surface without boundary with a discrete set $X$ of conical singularities such that the underlying metric space structure of $S$ is complete. Let $x\in X$ be a singular point. We define
\[l_x = \inf \{l' >0\mid x\text{ can be connected to a point in $X$ by a geodesic arc of length $l'$}\}.\]
Note that in the definition of $l_x$, we also allow geodesic loops based at $x$. If $S$ has has only one singular point $x$ and no geodesic loops based at $x$, i.e. when the definition of $l_x$ given above becomes void, then we make the convention $l_x = \pi$. 

\medskip
Let $\Sigma$ be a Troyanov's football with singular points $y$ and $y'$ of angles equal to the angle of $S$ at $x$. Our goal is the following: We show that, for any $r<l_x$, there exists a local isometry $f : \Sigma_r\rightarrow S$  which sends $y$ to $x$. When $l_x\ge\pi$, we argue that $S$ can have  at most two conical singular points.

\begin{prop}\label{R} 
There exists a family of local isometries $f_r : \Sigma_r\rightarrow S$, $f_r(y) = x$, indexed by $r\in[0,\min\{l_x,\pi\})$, such that, for $s>r$, $f_{s}|_{\Sigma_r} = f_r$.
\end{prop}

\begin{proof}Let $L = \min\{l_x,\pi\}$. We start by showing that the set \[R = \left\{ r\in[0,L) \mid \exists \text{ a local isometry}  f : \Sigma_r\rightarrow S \text{ such that } f(y) = x\right\}\] is equal to $[0,L)$.

\medskip
We prove that $R$ is a closed subset of $[0,L)$. Let $(r_k)_{k\in\mathbb{N}}$ be an increasing sequence in $R$ which converges to some number $r$ in $[0,L)$. Let $f_{r_k} : \Sigma_{r_k} \rightarrow S$, $f_{r_k}(y) = x$, be a sequence of local isometries indexed by $k\in\mathbb{N}$. We can extend the domain of $f_{r_k}$ to $\Sigma_r$ by keeping $f_{r_k}$ constant along the geodesics in $\Sigma_r - \text{int }(\Sigma_{r_k})$ orthogonal to the boundary $\partial\Sigma_{r_k}$. Therefore, we have an equicontinuous family of maps $f_{r_k}: \Sigma_r\rightarrow S$, for $k\in\mathbb{N}$.  Using Arzel\'a-Ascoli Theorem, by passing to a subsequence we can assume that the sequence $(f_{r_k})$ converges to a limit $f:\Sigma_r\rightarrow S$. In the following, we prove that $f$ is a local isometry which in turn proves that $R$ is closed.

We first prove that $f$ is a local isometry in the interior of $\Sigma_r$. We take any point $p\ne y$ in the interior of $\Sigma_r$. Let $q = \lim_{k\rightarrow\infty} f_{r_k}(p)\in S$. Since $0 < d_S(x, q) < L$, $q$ is a regular point of $S$, and hence, there exists $\pi/2>\delta>0$  such that the $\delta $-neighborhood $V$ of $q$ in $S$ is a spherical disk of radius $\delta$. We can choose $k_0\in\mathbb{N}$ and an $\delta/2$-neighborhood $U$ of $p$ in $\Sigma$ such that, for $k\ge k_0$, $U\subset \Sigma_{r_k}$ and $d_S(q, f_{r_k}(p))<\delta/2$. We may also assume that $U$ is isometric a spherical disk. It is clear that, for $k\ge k_0$, $U$ maps into $V$ under $f_{r_k}$. Since $V$ has radius $<\pi/2$, $V$ does not contain any closed geodesic. Moreover, $U$ is geodesically convex. Therefore, $f_{r_k}|_U$ is an isometry onto it's image. As a result, $f|_U$ is an isometric embedding into $S$.

We now show that $f$ is a local isometry at any boundary point $c\in\partial\Sigma_r$. The boundary $C_r = \partial\Sigma_r$ is a round circle in $\Sigma$. Let $c\in C$ and $U$ be an $\delta$-neighborhood of $c$ in $\Sigma$, for some $\delta<\min\{r,L-r\}$. Note that $f(c)$ is a regular point because we assumed that a geodesic from $x_i$ to any singular point has length at least $l_x>r$. Let $V$ be the $\delta$-neighborhood of $f(c)$ in $S$. We can also assume that $U$ and $V$ are spherical disks with centers $c$ and $f(c)$ respectively. From above, we know that $f$ maps $U_1 = U\cap \Sigma_r - \partial\Sigma_r$ locally isometrically into $V$. We show that  $f|_{U_1}$ is an isometric embedding.
If $U_1$ is geodesically convex, which happens precisely when $r<\pi/2$,  this follows as before. But this argument fails when $r>\pi/2$. In this case, let $f(t_1) = f(t_2)$, for some $t_1,t_2$ in $U_1$. We can find a third point $t_3$ in $U_1$ such that, for $i=1,2$, $t_i$ and $t_3$ can be connected by an (unnormalized) geodesic segment $\gamma_i:[0,1]\rightarrow U_1$, $\gamma_i(0) = t_3$, $\gamma_i(1) = t_i$. The geodesics  $f\circ \gamma_1$ and $f\circ\gamma_2$ in $V$ both have same initial and final points. Since $df$ is injective at all points of $U_1$, $\gamma'_1(0) = \gamma'_2(0)$ which then implies, after suitable renormalization, that either $\gamma_1 \subset \gamma_2$ or $\gamma_1 \supset\gamma_2$. Therefore, $t_1$ and $t_2$ must be joined by a geodesic segment in $U_1$, say $\gamma_3$. But then $f\circ\gamma_3$ is a closed geodesic in $S$ which forces $t_1=t_2$ because $V$ has radius at most $\pi/2$. As a result, $f|_{U_1}$ is injective i.e. an embedding. In this case, $f|_{U_1}$ can be extended to an embedding of $U$ in $V$ i.e. $f$ is a local isometry at $c$. Thus  $R$ is closed in $[0,L)$.

\medskip
Next we show that $R$ is also an open subset of $[0,L)$. It is clear that if $r\in R$, then the interval $[0,r]$ is also a subset of $R$. If for all $r\in R$ there exists some $s>r$ such that $s\in R$, then $r$ is an interior point of $R$. This shows that $R$ is an open subset of $[0,L)$.

\begin{lem}
Let $f : \Sigma_r\rightarrow S$ be a local isometry such that $f(y) = x$. Then, there exists some $s>r$ and a local isometry $\tilde{f} : \Sigma_s \rightarrow S$ such that $\tilde{f}|_{\Sigma_r} = f$.
\end{lem}

\begin{proof}
Let $c\in C_r$ be any point in the boundary $C_r = \partial \Sigma_r$. Since $f$ is a locally isometry, there exists a neighborhood $B$ of $c$ in $\Sigma$ and an embedding $f_1: B\rightarrow S$ such that $f|_{B\cap \Sigma_r} = f_1|_{B\cap \Sigma_r}$. The maps $f$ and $f_1$ patch together to  extend $f$ on the larger domain $\Sigma_r \cup B$ and such an extension is unique because our maps are analytic.

Let $c_1,\dots, c_p$ be points on $C_r$ with the following properties: (i) Each point $c_j$ has a neighborhood $B_j$ in $\Sigma$ such that $f$ can be extended to a local isometry $\tilde{f}_j : B_j \cup \Sigma_r\rightarrow S$. (ii) $\bigcup_{j=1}^p B_j$ covers the circle $C_r$. (iii) $(B_j\cap B_i)\cup \Sigma_r$ is connected. For example, we could choose $B_j$ to be a small disk centered at $c_j$. Clearly, (ii) implies that $\Sigma_r \cup \bigcup_{j=1}^p B_j$ contains $\Sigma_s$, for some $s>r$. Moreover, the uniqueness of the extensions $\tilde{f}_j$ and, (iii) ensure that they can be patched together to form an extension $\tilde{f} : \Sigma_r \cup \bigcup_{j=1}^p B_j \rightarrow S$ of $f$. Then $\tilde{f}|_{\Sigma_s}$ is an extension of $f$.
\end{proof}

The lemma combined with the fact that $R$ is closed in $[0,L)$ shows that any given local isometry $f_r : \Sigma_r \rightarrow S$, sending $y$ to $x$, can be extended to a local isometry $f_s : \Sigma_r\rightarrow S$, for any $L>s>r$. Therefore, given a local isometry $f_{r_0}: \Sigma_{r_0} \rightarrow S$ sending $y$ to $x$, for $0<r_0<L$, we have a sequence of local isometries $f_r : \Sigma_r \rightarrow S$, $f_r(y) = x$, indexed by $r\in[0,L)$ such that $f_s|_{\Sigma_r} = f_r$, for $s\ge r$. This completes the proof of the proposition.
\end{proof}

\begin{rem}The family of local isometries in the proposition may not be replaced by a family of isometries, as shown in the following example:
\end{rem}
\begin{exmp}Let $D$ be the closed upper hemisphere of the standard sphere $\mathbb{S}^2$. The boundary of $D$ is a geodesic. Let $x_1,x_2,x_3,x_4$ be evenly distributed points on the boundary as shown in the figure below. 
\begin{figure}[h]
\centering
\begin{tikzpicture}[line cap=round,line join=round,>=triangle 45,x=1cm,y=1cm]
        \clip(-3.26,-3.36) rectangle (2.98,2.56);
        \draw [line width=1.2pt] (-0.14,-0.35) circle (2.320086205294967cm);
        \draw (-0.44,2.56) node[anchor=north west] {$x_1$};
        \draw (2.42,-0.12) node[anchor=north west] {$x_2$};
        \draw (-0.26,-2.86) node[anchor=north west] {$x_3$};
        \draw (-3.18,-0.12) node[anchor=north west] {$x_4$};
        \draw (1.74,1.5) node[anchor=north west] {$a$};
        \draw (1.64,-1.86) node[anchor=north west] {$b$};
        \draw (-2.34,-1.86) node[anchor=north west] {$c$};
        \draw (-2.4,1.58) node[anchor=north west] {$d$};
        \draw [->,line width=1pt] (-1.974999492308493,1.0697101335228867) -- (-1.705522711622132,1.3622904658367068);
        \draw [->,line width=1pt] (1.5050005076915067,-1.9802898664771131) -- (1.7744772883778677,-1.687709534163293);
        \draw [->,line width=1pt] (1.839687411484705,0.8598089736850975) -- (1.630740307505552,1.149092646695041);
        \draw [->,line width=1pt] (-1.740312588515295,-2.050191026314902) -- (-1.9492596924944483,-1.7609073533049584);
        \draw (-0.36,-.12) node[anchor=north west] {$D$};
        \begin{scriptsize}
        	\draw [fill=black] (-2.459012431119667,-0.42057863920798977) circle (2pt);
        	\draw [fill=black] (-0.14,1.970086205294967) circle (2pt);
        	\draw [fill=black] (2.179895557675162,-0.3202577492605749) circle (2pt);
        	\draw [fill=black] (-0.09948065761460276,-2.6697323515640026) circle (2pt);
        \end{scriptsize}
\end{tikzpicture}
\end{figure}
By identifying the directed edges $a$ with $c$ and $b$ with $d$, we obtain a torus $T$ with a spherical metric having only one conical singularity $x$ of angle $4\pi$, and in this case, $l_x = \pi/2$ because the shortest geodesic loop based at $x$ has length $\pi/2$. Note that a local isometry $f_r : \Sigma_r\rightarrow S$ fails to be an injective map when $r>\pi/4$.
\end{exmp}

In the following, we assume that $l_x\ge \pi$. Let $\dot{\Sigma}$ denote the punctured sphere $\Sigma - \{y\}$. Using Proposition \ref{R}, we have a family of local isometries $f_r : \Sigma_r\rightarrow S$, $f_r(y) = x$,  indexed by $r\in[0,\min\{l_x,\pi\})$, such that, for $s>r$, $f_{s}|_{\Sigma_r} = f_r$. We construct a map
\[\dot{F} : \dot{\Sigma} \rightarrow S\]
by setting $\dot{F}(z) = f_r(z)$, where $r = d_\Sigma (y,z)$. It is clear that $\dot{F}$ is a local isometry, and hence, a Lipschitz continuous function with Lipschitz constant $1$. Since the codomain of $\dot{F}$ is a complete metric space,  $\dot{F}$ can be extended to a map $F:\Sigma\rightarrow S$.

\begin{prop}\label{surjective}
$F : \Sigma \rightarrow S$ constructed above is a surjective map.
\end{prop}

\begin{proof}
Since $F|_{\Sigma-\{y,y'\}}$ is a local isometry, image $S'$ of $\Sigma-\{y,y'\}$ under $F$ is open in $S$. Adjoining the points $x = F(y)$ and $F(y')$ compactifies $S'$. $F(y) = x$ is an interior point of the image $F(\Sigma)$. We prove that $F(y')$ is also an interior point of the image $F(\Sigma)$. Let $D$ be a neighborhood of $F(y')$, homeomorphic to $\mathbb{R}^2$. $D\cap S'$ is a nonempty open set in $D$ which becomes a closed subset when we adjoin $F(y')$. So, $D\cap S'$ and complement of $(D\cap S')\cup \{F(y')\}$ are disjoint open subsets whose union is $D - \{F(y')\}$. Since $D - \{F(y')\}$ is connected, the complement of $(D\cap S')\cup \{F(y')\}$ must be empty, i.e. $D\cap S' \subset S'$. This means that $F(\Sigma)$ contains $D$.

Hence, the image $F(\Sigma)$ is open in $S$. Since $S$ is connected, $F(\Sigma) = S$.
\end{proof}

We close this section by completing the proof of Theorem \ref{thm:general}.

\begin{proof}[Proof of Theorem \ref{thm:general}] Suppose that $S$ has no geodesic loop based at $x$ of length $\le\pi$. If $\displaystyle\min_{x\ne x'\in X}\{d_S(x,x')\}$ $ \ge \pi$, then $l_x\ge \pi$. The image $F(\dot{\Sigma})$ constructed above is a regular open subset of $S$, i.e. it contains no conical singular point of $S$. From Proposition \ref{surjective}, this image misses at most two points of $S$, namely, $F(x)$ and $F(x')$. Therefore, $S$ can have at most two singular points. Compactness of $S$ follows from the surjectivity of $F$.
\end{proof}

\section{Proof of Theorem \ref{thm:main}}

Let $S$ be a closed, genus zero surface with a spherical metric $g$ with $n$ conical singularities $x_1,\dots, x_n$ of angles $2\pi\t_1,\dots,2\pi\t_n$ respectively, where all $\theta_i$'s are non-integers and the tuple $(\theta_1, \dots,\theta_n)$ satisfies (\ref{H'}). An important feature of the metric $g$ on $S$ is that the holonomy representation is coaxial, which follows from Theorem \ref{panov1}. 

\begin{prop}\label{prop:distance}
Let $\gamma$ be a geodesic arc in $S$ with singular endpoints. Then the length of $\gamma$ is an integral multiple of $\pi$. In particular, for $i\ge 2$, $d_S(x_1,x_i)\ge \pi$ and $S$ has no geodesic loop based at $x_1$ of length shorter than $\pi$.
\end{prop}

\begin{proof}
We shall use Lemma 2.10 of \cite{mondello2016spherical} which states the following.
\begin{lem}
Let $S'$ be a sphere with a spherical metric with conical singularities  having coaxial holonomy. Suppose that ${\gamma}'$ is a geodesic arc from $x_1'$ to $x_2'$, where $x_1'$ and $x_2'$ are distinct conical singularities of angles $\not\in 2\pi\mathbb{N}$. Then the length of ${\gamma}'$ is an integral multiple of $\pi$.
\end{lem}

The lemma shows that the length of $\gamma$ is an integral multiple of $\pi$ if $\gamma$ has distinct endpoints. So, we can assume that $\gamma$ is a geodesic loop based at $x_1$. The complement of $\gamma$ in $S$ is comprised of connected components among which there are exactly two components, say $U_1$ and $U_2$, whose boundary contains the point $x_1$. For $i=1,2$, let $y_i\in U_i$ be a regular point. Let $\rho: \tilde{S}\rightarrow S$ be the two-fold branched covering, branched at the points $y_1$ and $y_2$, where $\tilde{S}$ is a closed, genus zero surface. By pulling back the metric of $S$, $\tilde{S}$ gets a natural spherical metric with a set of conical singularities $\rho^{-1}(X\cup\{y_1,y_2\})$ where $X = \{x_1,\dots,x_n\}$. Moreover, the holonomy of $\tilde{S}$ is coaxial which follows from holonomy representation
\[
\begin{tikzcd}
\pi_1(\tilde{S} - \text{singularities}) \arrow[r,"\rho_*"] & \pi_1({S} - X\cup \{y_1,y_2\}) \arrow[r, "\phi"] & SO(3),
\end{tikzcd}
\]
where $\phi: \pi_1({S} - X\cup \{y_1,y_2\}) \rightarrow SO(3)$ is the holonomy representation associated to $S$ with conical singularities $y_1,y_2,x_1,x_2\dots,x_n$. Here we treat the regular points $y_1$ and $y_2$ of $S$ as singular points of angle $2\pi$. The loop $\gamma$ can be lifted to a geodesic path $\tilde{\gamma}$ with distinct endpoints which are precisely the pre-images of $x_1$. We have
\[ \text{length of } \gamma =  \text{length of } \tilde{\gamma},\]
and the later is an integral multiple of $\pi$ which follows from the lemma.
\end{proof}

\begin{proof}[Proof of Theorem \ref{thm:main}]
Suppose that, for $n\ge 3$, $\pmb{\theta} = (\theta_1,\dots, \theta_n)\in\mathbb{R}_+^n - \mathbb{Z}^n$ satisfies (\ref{P}) and (\ref{H'}). If $\pmb{\theta}$ is admissible, then there is a sphere with a spherical metric with conical singularities $x_1,\dots,x_n$ such that the angle at $x_i$ of $g$ is $2\pi\theta_i$. By Proposition \ref{prop:distance}, $d_S(x_1, x_i) \ge \pi$, for $i\ge 1$, and there is no geodesic loop  based at $x_1$ of length shorter than $\pi$. Using Theorem \ref{thm:general}, we get $|X| \le 2$. This is a contradiction.
\end{proof}

\bibliographystyle{plain}

\Address
\end{document}